\documentclass[12pt]{amsart}
\usepackage[utf8]{inputenc}
\usepackage[margin=1in]{geometry}
\usepackage{tikz}
\usetikzlibrary{calc}
\usepackage{cleveref}
\usepackage{caption}
\usepackage{graphicx}
\graphicspath{{/images}}
\usepackage{mathtools}

\usetikzlibrary{external}
\usepackage{amsthm,amssymb,bbm}

\newtheorem{theorem}{Theorem}[section]
\newtheorem{cor}[theorem]{Corollary}
\newtheorem{prop}[theorem]{Proposition}
\newtheorem{lemma}[theorem]{Lemma}

\theoremstyle{definition}
\newtheorem{example}[theorem]{Example}
\newtheorem{defn}[theorem]{Definition}
\newtheorem{q}[theorem]{Question}
\newtheorem{remark}{Remark}

\newcommand{\newword}[1]{\textbf{\textit{#1}}}

\newcommand{\argmax}{\mbox{argmax}}
\newcommand{\RR}{\mathbb{R}}
\newcommand{\qq}{\mathbbm{1}}

\newcommand{\lca}{\mbox{LCA}}

\newcommand{\ta}{\mathcal{T}_{[n]}}
\newcommand{\tb}{\mathcal{T}_{[n]}^\textup{planar}}
\newcommand{\tc}{\mathcal{T}_{[n], v}^\textup{planar}}
\newcommand{\td}{\mathcal{T}_{n}^\textup{planar}}
\newcommand{\te}{\mathcal{T}_{n, v}^\textup{planar}}

\newtheorem{innercustomthm}{Theorem}
\newenvironment{mythm}[1]
  {\renewcommand\theinnercustomthm{#1}\innercustomthm}
  {\endinnercustomthm}
\newtheorem{innercustomcor}{Corollary}
\newenvironment{mycor}[1]
  {\renewcommand\theinnercustomcor{#1}\innercustomcor}
  {\endinnercustomcor}

\title{Classifying Tree Topology Changes along Tropical Line Segments}
\author{Shelby Cox}
\date{\today}

\begin{document}

\begin{abstract}
    The space of phylogenetic trees arises naturally in tropical geometry as the tropical Grassmannian. Tropical geometry therefore suggests a natural notion of a tropical path between two trees, given by a tropical line segment in the tropical Grassmannian. It was previously conjectured that tree topologies along such a segment change by a combinatorial operation known as Nearest Neighbor Interchange (NNI). We provide counterexamples to this conjecture, but prove that changes in tree topologies along the tropical line segment are either NNI moves or ``four clade rearrangement" moves for generic trees. In addition, we show that the number of NNI moves occurring along the tropical line segment can be as large as $n^2$, but the average number of moves when the two endpoint trees are chosen at random is $O(n (\log n)^4)$. This is in contrast with $O(n \log n)$, the average number of NNI moves needed to transform one tree into another.
\end{abstract}

\maketitle

\section{Introduction}



Phylogenetic trees are leaf-labelled, metric trees that encode evolutionary relationships between a collection of organisms. In this paper, phylogenetic trees are also rooted. The root represents the common ancestor of the organisms, the leaves represent present day organisms, and the lengths of the edges represent the amount of mutation that has occurred. The space of phylogenetic trees with $n$ present day organisms (i.e. leaves) is a topological space, and we will denote it by $\mbox{Tree}_n$.

The space $\mbox{Tree}_n$ has been realized with various combinations of topological and metric structures: as a CAT(0) metric space (Billera, Holmes, and Vogtman \cite{bhv}), as a tropical variety (Speyer and Sturmfels \cite{speyer2004tropical}, Ardila and Klivans \cite{berg_ak}), the edge-product space (Kim \cite{kim2000slicing}, Moulton and Steel \cite{moulton_steel_oranges}, and Gill et al. \cite{gill2008regular}), and the Wald space (Gabra et al. \cite{wald}). In this paper, we focus on a specific realization of $\mbox{Tree}_n$ as a tropical variety, and defer to \cite[Section 1]{wald} for their excellent review of constructions of $\mbox{Tree}_n$, and related results.

Under a certain embedding, $\mbox{Tree}_n$ is endowed with the structure of a tropical variety -- \newword{the space of ultrametrics} -- and in this realization $\mbox{Tree}_n$ is \newword{tropically convex}, meaning that any point on the \newword{tropical line segment} between two phylogenetic trees on $n$ leaves can itself be interpreted as a phylogenetic tree on $n$ leaves. Under this framework, Lin and Yoshida studied sample means \cite{rudy-fw}, and Yoshida et al. developed a version of principal component analysis \cite{trop-pca}. In these results, tropical polytopes in tree space play a critical role, and yet it is poorly understood which trees appear in a tropical polytope with tree vertices. Therefore, the first goal of the paper is to answer the following question:

\begin{q}
    How do tree topologies (i.e. tree structures) change along the tropical line segment?
\end{q}

It was conjectured previously that the trees along the tropical line segment change by Nearest Neighbor Interchange (NNI). This is not quite true. For example, in \cite{myown} it was shown that the tropical line between two trees on four leaves may pass through the star tree with probability greater than 0, and this is not an NNI move. However, the behavior does not get much worse.

The tropical line segment is a concatenation of classical line segments. The tree topology can change at the points where the classical line segments connect, but not on the interior of any of the classical line segments. We call these points \newword{turning points} (whether or not the tree topology changes). In Section \ref{section:tnni}, we prove the following result:

\begin{mythm}{\ref{thm:turning-pts}}
    Tree topology changes only occur at turning points; at each turning point, one of three things can happen:
    \begin{enumerate}
        \item (Nearest Neighbor Interchange) One internal vertex has three children, see \Cref{fig:singleNNI}.
        
        \begin{figure}
            \centering
            \includegraphics{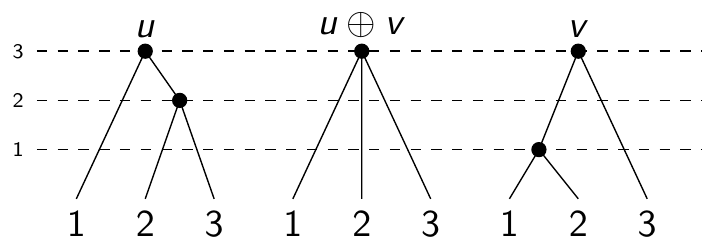}
            \caption{The turning points on a tropical line segment with a single NNI move.}
            \label{fig:singleNNI}
        \end{figure}
        
        \item (Four Clade Rearrangement) One internal vertex has four children, see \Cref{fig:fourClade}.
        
        \begin{figure}
            \centering
                \includegraphics{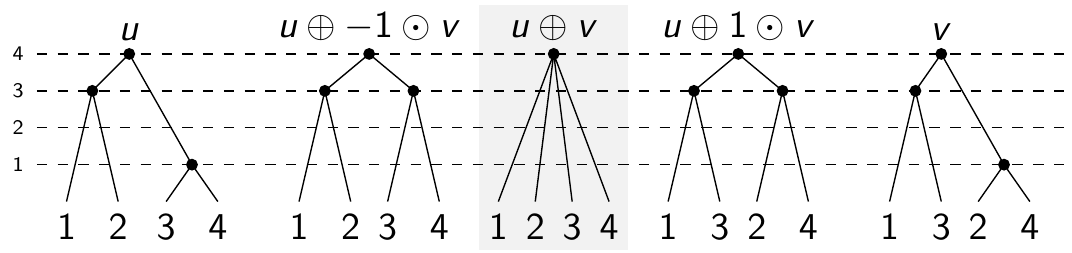}
            \caption{The turning points on a tropical line segment with a four clade rearrangement.}
            \label{fig:fourClade}
        \end{figure}
        
        \item (No Topology Change) All internal vertices have two children, see \Cref{fig:noTop}.
        
        \begin{figure}
            \centering
            \includegraphics{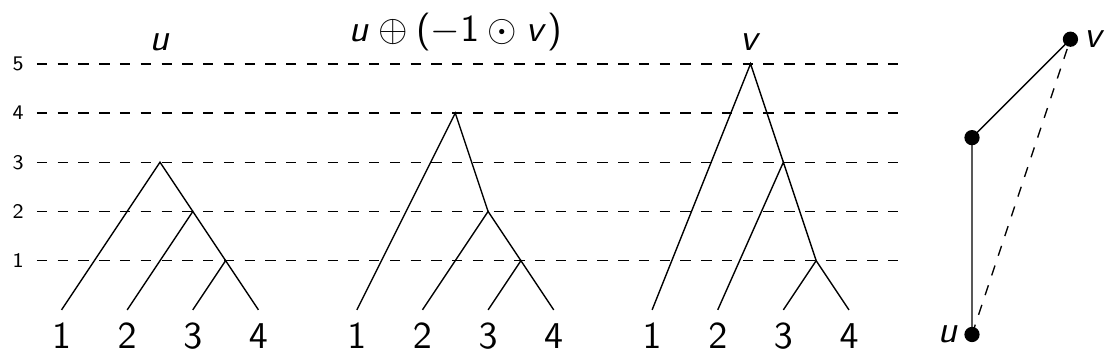}
            \caption{The turning points on a tropical line segment with constant tree topology (left) and the tropical line segment (right solid), and straight line segment (right dashed) between $u$ and $v$.}
            \label{fig:noTop}
        \end{figure}
    \end{enumerate}
\end{mythm}

The four clade rearrangement move can be achieved by three nearest neighbor interchange moves, so it makes sense to ask how many NNI moves occur along the tropical line segment between two general trees -- each turning point can contribute 0, 1, or 3 NNI moves. We will refer to the number of NNI moves occurring along the tropical path as the \newword{tropical NNI number} (counting a four-clade rearrangement as three NNI moves). 

\begin{q}
    What is the tropical NNI number between a random pair of trees on $n$ leaves? How does it compare to the \newword{NNI distance}, i.e. the minimal number of NNI moves required to transform one tree into the other?
\end{q}

By a pair of \textit{randomly chosen phylogenetic trees} on $[n]$ leaves, we mean that the topologies of the trees are uniformly randomly chosen from the (finitely many) binary tree topologies on $[n]$ leaves, and the corresponding tree metrics are sufficiently general, in a way that is made precise in \Cref{section:redef-tp}. For a random pair of phylogenetic trees on $n$ leaves, it is known that the average NNI distance is $O(n \log(n))$, and the minimal number of NNI moves is NP-hard to compute \cite{DasGupta2016}. \Cref{thm:turning-pts} implies that the number of turning points is an upper bound on triple the number of NNI moves on the tropical line segment. First we provide an example where the tropical NNI number can be much larger than the minimal number of NNI moves.

\begin{mythm}{\ref{thm:bad-example}}
    There exist generic pairs of trees, $T_1$, $T_2$, on $n$ leaves, for each $n$, so that the tropical NNI number from $T_1$ to $T_2$ is $\binom{n-1}{2}$.
\end{mythm}

Then we prove an upper bound for the average case.

\begin{mycor}{\ref{cor:avg-NNI-bound}}
    The expected number of NNI moves occurring along the tropical line segment between two randomly chosen trees on $n$ leaves is $O(n (\log(n))^4)$.
\end{mycor}


This paper is structured as follows. In Section \ref{background} we introduce background and conventions on tree spaces, tropical mathematics, and nearest neighbor interchange (NNI). In Section \ref{section:redef-tp}, we derive a non-standard definition of turning points, which will be useful throughout the rest of the paper. In Section \ref{section:tnni} we classify tree topologies changes occurring along the tropical line segment between two generic trees, proving \Cref{thm:turning-pts}. In Section \ref{section:bad-example}, we describe a family of generic pairs of trees whose tropical NNI number far exceeds the expected tropical NNI number (proving \Cref{thm:bad-example}). In Section \ref{section:bound}, we bound the expected NNI distance along the tropical line segment between two randomly chosen trees on $n$ leaves (proving \Cref{thm:bound}). 

\section{Background}\label{background}

\subsection{Tropical Semi-ring}
The tropical semi-ring turns out to be a good place to work with trees. We will work in the tropical max-plus semi-ring $\left(\RR \cup \{ -\infty \}, \oplus, \odot\right)$, where tropical addition and multiplication are given by $\max$ and classical addition, respectively:
$$a \oplus b := \max\{ a, b \}, ~~~ a \odot b := a + b.$$
Note that $-\infty$ is the identity element for addition, and $0$ is the identity element for multiplication. Tropical division is classical subtraction; tropical additive inverses usually do not exist. Geometrically, we will work in the tropical projective space $\RR^n/\RR \qq$, where $\qq = (1,1,\ldots,1)$ is the all 1's vector.

\begin{defn}[\cite{Develin2004}]\label{defn:tr-conv}
    The \newword{tropical convex hull} of $u_1,\ldots,u_k \in \RR^n$ is the collection of all tropical linear combinations of the $u_i$:
    $$\mbox{tconv}(u_1,\ldots,u_k) = \{ \lambda_1 \odot u_1 \oplus \cdots \oplus \lambda_k \odot u_k \mid \lambda_i \in \RR \}.$$
\end{defn}


Note that if $u_i^\prime = \alpha_i \odot u_i$, then
\begin{align*}
    \mbox{tconv}(u_1,\ldots,u_k) &= \{ \lambda_1 \odot u_1 \oplus \cdots \oplus \lambda_k \odot u_k \mid \lambda_i \in \RR \} \\
    &= \{ (\lambda_1 - \alpha_1) \odot (\alpha_1 \odot u_1) \oplus \cdots \oplus (\lambda_k - \alpha_k) \odot (\alpha_k \odot u_k) \mid \lambda_i \in \RR \} \\
    &= \{ \lambda_1^\prime \odot u_1^\prime \oplus \cdots \oplus \lambda_k^\prime \odot u_k^\prime \mid \lambda_i^\prime \in \RR \} \\
    &= \mbox{tconv}(u_1^\prime,\ldots,u_k^\prime)
\end{align*}
i.e., the tropical convex hull is well defined in the tropical projective torus, $\RR^m/\RR \qq$.

A tropical line segment is the tropical convex hull of two points $u$, $v$. It is a concatenation of at most $n-1$ classical (i.e. Euclidean) line segments, and it can be computed with the algorithm in \Cref{defn:line-algo}.

\begin{defn}[Algorithm for the Tropical Line Segment]\label{defn:line-algo}
    \begin{enumerate}
        \item[]
        \item \textbf{input:} $u,v \in \RR^n$.
        \item $\Lambda(u, v) = [ u_i - v_i \mid i \in [n]]$, sorted from least to greatest.
        \item $w_i = u \oplus (\lambda_i \odot v)$, where $\lambda_i$ is the $i$th smallest element of $\Lambda$.
        \item \textbf{output:} the concatenation of Euclidean line segments from $w_i$ to $w_{i+1}$.
    \end{enumerate}
\end{defn}

\begin{defn}\label{defn:turning}
    In the algorithm above, $\Lambda(u, v) = \{ u_i - v_i \mid i \in [n] \}$ is the \newword{set of turning point scalars}, and $w_i$ is a \newword{turning point} for the tropical line segment from $u$ to $v$.
\end{defn}

\begin{example}
    Let $u = (3,3,1)$ and $v = (3,2,3)$. Then $\Lambda = [-2,0,1]$, and the turning points on the tropical line are:
    \begin{align*}
        u \oplus (-2 \odot v) &= (3,3,1) \oplus (1,0,1) = (3,3,1) &&= u \\
        u \oplus (0 \odot v) &= (3,3,1) \oplus (3,2,3) = (3,3,3) &&= (0,0,0) \\
        u \oplus (1 \odot v) &= (3,3,1) \oplus (4,3,4) = (4,3,4) &&= v.
    \end{align*}

    This tropical line segment is depicted in \Cref{fig:tropline}.
    
    \begin{figure}
        \centering
        \includegraphics{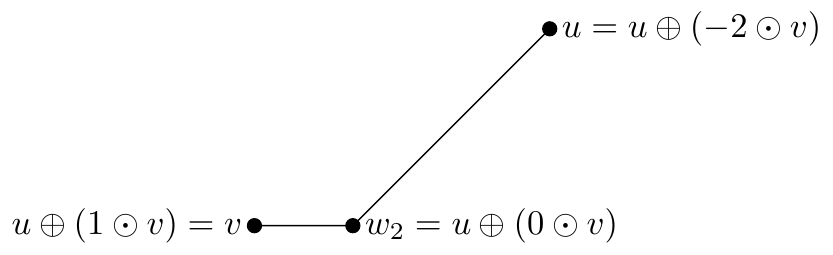}
        \caption{The tropical line segment from $u$ to $v$.}
        \label{fig:tropline}
    \end{figure}
\end{example}

\subsection{Phylogenetic Trees and Ultrametrics}

\begin{defn}\label{defn:tree}
    A \newword{phylogenetic tree} on $n$ leaves is a rooted\footnote{In general, phylogenetic trees need not be rooted.}, metric tree, with positive lengths on internal edges and leaf label set $[n] = \{ 1, 2, \ldots, n \}$. We will denote the root of the tree (an internal vertex) by $\rho$.
\end{defn}

\begin{defn}\label{defn:lca}
    Given two leaves $i,j$ of $T$, the \newword{least common ancestor} of $i$ and $j$ is the internal vertex in the intersection of the paths $i \rightsquigarrow j$, $\rho \rightsquigarrow i$, $\rho \rightsquigarrow j$. We denote the least common ancestor of $i,j$ in $T$ by $\mbox{LCA}_T(i,j)$. 
\end{defn}

Given a metric tree $T$ and two of its vertices, $a$, $b$, we define the distance from $a$ to $b$ to be the sum of the lengths of edges on the unique path from $a$ to $b$. Denote this distance by $d_T(a,b)$.

We are especially interested in \textit{equidistant trees}.

\begin{defn}\label{defn:equidistant}
    A phylogenetic tree is an \newword{equidistant tree} if each leaf is the same distance $h$ from $\rho$.
\end{defn}

For rooted equidistant trees, the height of an internal vertex is well-defined and it will be useful to have a notation for the height of an internal vertex.

\begin{defn}\label{defn:height}
    Let $x$ be an internal vertex of $T$. For any $i, j \in [n]$ such that $\lca_T(i,j) = x$,
    $$h_T(x) = \frac12 d_T(i,j)$$
\end{defn}

We may associate to each phylogenetic tree $T$ (equidistant or not) a vector which records the distances between leaves. Denote this vector by $u^T$ (or just $u$):
$$u = u^T = (d_T(1,2), d_T(1,3), \ldots, d_T(n-1,n)) \in \RR^{\binom{n}{2}}.$$
This is called the \newword{tree metric} for $T$. We will often refer to $u$ and $T$ interchangeably. So $i \in [n]$ may be called a leaf of $u$, $\mbox{LCA}_u(i,j) = \mbox{LCA}_T(i,j)$, and $h_u(x) = h_T(x)$.

Conversely, given a vector in $\RR^{n \choose 2}$, we want to know when it is a vector coming from the leaf distances on a phylogenetic tree, and when it is a vector coming from the leaf distances on an equidistant tree. These are the four point and three point conditions, respectively:

\begin{theorem}[Four Point Condition \cite{textbook}]
    $u \in \RR_{\geq 0}^{n \choose 2}$ is the tree metric for a phylogenetic tree $T$ if $\max\{ u_{ij} + u_{kl}, u_{ik} + u_{jl}, u_{il} + u_{jk} \}$ is achieved at least twice for all distinct $i,j,k,l \in [n]$.
\end{theorem}
    
\begin{theorem}[Three Point Condition \cite{textbook}]
    $u \in \RR_{\geq 0}^{n \choose 2}$ is the tree metric for an equidistant tree $T$ if $\max\{ u_{ij}, u_{ik}, u_{jk} \}$ is achieved at least twice for all distinct $i,j,k \in [n]$.
\end{theorem}

It will be convenient when counting trees later to add a leaf labelled 0 adjacent to the root $\rho$, so that all internal vertices of a generic tree have degree exactly 3. However, we will not include distances to the 0 leaf in the tree metric vector.

\begin{defn}\label{defn:ultrametric}
    An \newword{ultrametric} is a vector in $\RR^{n \choose 2}$ satisfying the three-point condition.
\end{defn}

\begin{remark}
    An ultrametric $u \in \RR_{\geq 0}^{n \choose 2}$ is a metric on $[n]$.
\end{remark}

We will often refer to an equidistant tree by giving its tree metric as a vector (an ultrametric) in $\RR^{n \choose 2}$. Furthermore, it is natural to view ultrametrics in a tropical projective space, i.e. adding $\lambda \qq$ to the ultrametric vector doesn't change the structure of the equidistant tree. When we add $\lambda \qq$ to the ultrametric vector, the change in the corresponding equidistant tree is a change in the external (leaf-adjacent) edge lengths.

The following theorem is essential to our study of tropical tree space:
\begin{theorem}[\cite{textbook}]
    The space of ultrametrics is tropically convex.
\end{theorem}

This means that the tropical line between two ultrametrics (or equidistant trees) stays in the space of ultrametrics.

\subsection{The Space of Phylogenetic Trees}

By \newword{tree topology}, we mean the combinatorial structure of a tree as a leaf-labelled graph. We will define it in three ways: using edge splits (BHV definition), using the argmax function (metric definition), and finally using a fan structure on the tropical Grassmannian (tropical definition). Starting with the BHV definition:

\begin{defn}\label{defn:edge-split}
    Given a phylogenetic tree $T$ and an edge $e$, the edge split for $T$ at $e$ is the partition that $e$ induces on the leaves of $T$.
\end{defn}

\begin{defn}[BHV Tree Topology]\label{defn:topo-bhv}
    Two trees have the same \newword{topology} if they have the same set of edge splits.
\end{defn}

We can also use the tree metric to determine the tree topology. If $T$ is a tree, $u$ is the corresponding ultrametric, and $S \subset \binom{[n]}{2}$, define $\argmax^1_u$ to be the subset of indices in $S$ which maximize $u_{ij}$. Then define $\argmax_u^m \{ S \}$ to be the subset of indices of $S$ which achieve the $m$th largest value of $u_{ij}$ among indices in $S$. Formally,  we define the \newword{argmax} function below.

\begin{defn}\label{defn:argmax}
    Let $\argmax_u ( S ) = \{ s \in S : u_s = \max_{x \in S} u_x \}$ for any $S \subseteq \binom{[n]}{2}$.
\end{defn}

\begin{defn}\label{defn:topo-argmax}
    Two phylogenetic trees $u$, $v$ have the same \newword{topology} if
    $$\argmax_u\{ ij, ik, jk \} = \argmax_v\{ ij, ik, jk \}$$
    for all $i < j < k \in [n]$.
\end{defn}

Finally, we define tree topology using a fan structure on the tropical Grassmannian. A tropicalized linear space has many fan structures. Here we consider the fan structure first described by Speyer and Sturmfels in \cite{speyer2004tropical}. We will say that two trees represented by ultrametrics $u$, $v$ have the same \newword{topology} if they lie in the relative interior of the same cone in the fan. The various definitions for tree topology are summarized below.

\begin{defn}\label{defn:topos}
    Two equidistant trees $T_1$, $T_2$ on $n$ leaves, represented by $u_1, u_2 \in \mathcal{U}_n$, have the same (coarse) topology if any of the following equivalent statements hold:
    \begin{enumerate}
        \item $u_1$, $u_2$ both lie in the relative interior of the same cone in the coarse fan structure on $\mathcal{U}_n$, or
        \item $\argmax_{u^1} \{ ij, ik, jk \} = \argmax_{u^2} \{ ij, ik, jk \}$ for every triple $i < j < k \in [n]$, or
        \item if $T_1$, $T_2$ are isomorphic as leaf-labelled trees, or
        \item if $T_1$, $T_2$ determine the same edge splits (BHV definition).
    \end{enumerate}
\end{defn}

We define the space of phylogenetic trees.

\begin{defn}\label{defn:Tn}
    The \newword{space of phylogenetic trees} is the collection of equidistant phylogenetic trees on $n$ leaves; equivalently, it is the collection of metrics on $n$ points satisfying the 3-point condition:
    \begin{align*}
        \mbox{Tree}_n^{\text{tr}} &= \{ (d_T(i,j))_{i < j} \in \RR^{n \choose 2}/\RR \qq \mid T \text{ a phylogenetic tree on } n \text{ leaves} \} \\
        &= \{ u \in \RR^{n \choose 2}/\RR \qq \mid \max(u_{ij}, u_{ik}, u_{jk}) \text{ is achieved at least twice for } i < j < k. \}
    \end{align*}
\end{defn}

\begin{theorem}[\cite{speyer2004tropical}, Theorem 4.2]
    Tree space is a union of simplicial cones, with each cone containing all the points representing a fixed tree topology $\tau$.
\end{theorem}

\begin{remark}
    Another fan structure on the tropical Grassmannian is described by Ardila and Klivans in \cite{berg_ak}. It is finer than the fan structure we described above; two trees $u$, $v$ lie in the same cone of the finer structure and have the same \textit{ranked tree topology} if they are combinatorially isomorphic, and the chronological order of all $n-1$ internal vertices is the same in both trees, or equivalently if
    $$\argmax^k_u \binom{[n]}{2} = \argmax^k_v \binom{[n]}{2}, \text{ for all } k \in \mathbb{N},$$
    where $\argmax^k_u (S)$ denotes the elements of $S$ where $u$ achieves its $k$th largest value over elements in $S$.
    
    Note that in \Cref{fig:fourClade}, the second tree, $w = u \oplus -1 \odot v$, lies in a codimension 1 cone in the fine fan structure (since the tree is binary branching, but two internal vertices are at the same height). More precisely, the line segment between $w$ and $u \oplus v$ (the second and third trees in \Cref{fig:fourClade}) lies in the intersection of two cells representing distinct \textbf{ranked} tree topologies: $w_{12} \leq w_{34} \leq w_{13} = w_{14} = w_{23} = w_{24}$ and $w_{34} \leq w_{12} \leq w_{13} = w_{14} = w_{23} = w_{24}$.
    
    The number of turning points is close to the NNI distance, because each turning point sits in a coarse codimension cone of at most two, hence uses at most three NNI moves. In the finer fan structure, the tropical line segment can pass through cones of codimension up to $n/2$, which indicates that many \textit{ranked} NNI may be required at a single turning point. Additionally, in the coarse fan structure, the tropical line segment only intersects non-maximal cones at points, but in the finer fan structure, the tropical line segment may intersect non-maximal cones in a whole classical line segment.
\end{remark}

\subsection{Nearest Neighbor Interchange}

One common way of transforming phylogenetic trees is Nearest Neighbor Interchange (NNI), which was introduced independently in \cite{NNIorigin1} and \cite{NNIorigin2}. NNI moves are all possible rearrangements of any three clades partitioning the leaves of the tree. An internal edge of a rooted tree partitions the leaves into three clades: $B$, $C$, and $D$. An NNI looks like one of the moves in \Cref{fig:NNI}.

\begin{defn}\label{defn:clade}
    A \newword{clade} of a phylogenetic tree $T$ is the collection of all leaves descending from a given fixed internal vertex.
\end{defn}

\begin{figure}
    \centering

    \includegraphics[]{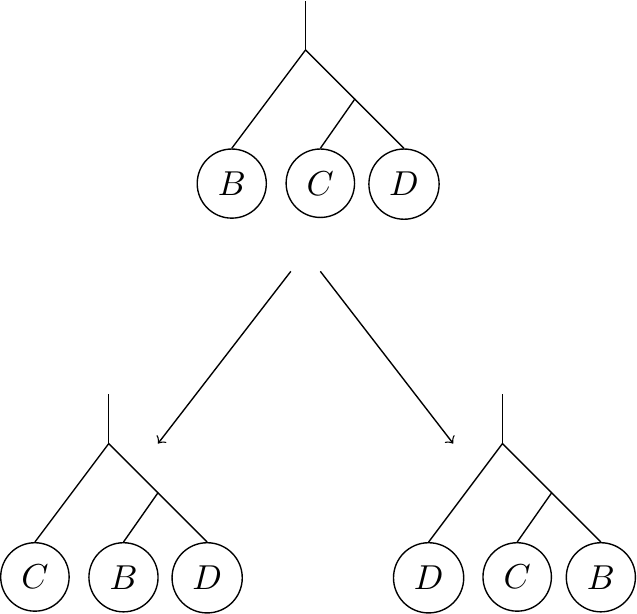}

    \caption{Possible NNI Moves}
    \label{fig:NNI}
\end{figure}

The NNI graph for $n$ leaves has vertices representing all tree topologies on $n$ leaves, and has edges between vertices whose tree topologies differ by a single NNI move. Note that in the NNI graph, we forget the edge lengths in the tree.

The NNI graph is a well-studied graph. In 1997, DasGupta et al showed that computing shortest paths in the NNI graph is NP-complete \cite{NNIdistances}. And \cite{someNNInotes} provides a polynomial time $O(n^2)$ algorithm that approximates the length of the shortest NNI path up to a factor of $4 \log(n) + 4$. The diameter of the NNI graph is known to be $O(n \log(n))$, and the expected distance between two randomly chosen vertices is also $O(n \log (n))$. We are particularly interested in whether the tropical NNI number is close to the polynomial time approximation to the NNI distance given in \cite{someNNInotes}.

Another type of move that can occur along the tropical line segment is the four clade rearrangement move, where four branches may come together at once. An example is given in \Cref{fig:fourClade}. This transformation in the tree topology can also be achieved by a sequence of three NNI moves.

\section{Redefining Turning Points for Trees}\label{section:redef-tp}

We begin by reinterpreting the turning point scalars in the tropical line segment algorithm (\Cref{defn:line-algo}) in terms of internal vertices of the endpoint trees, which will allow us to describe turning points without referring to the tree metrics.

\begin{defn}\label{defn:tp-red}
    Given a pair of phylogenetic trees $T_1, T_2$ on $n$ leaves, define the \newword{essential pairs} of $T_1, T_2$ to be
    $$\Pi(T_1, T_2) := \{ (x_1, x_2) \mid x_k \in T_k, \exists i, j \in [n], i \neq j \mbox{ s.t. } x_k = \lca_{T_k}(i,j) \}$$
\end{defn}

In the sections that follow, we will find that the results are much nicer if we consider trees with sufficiently generic ultrametrics. We qualify ``sufficiently generic ultrametric" with the definition of \newword{generic} below.

\begin{defn}\label{defn:generic}
    An equidistant tree $T$ on $n$ leaves with distance vector $u$ is \newword{generic} if any of the following equivalent conditions hold:
    \begin{enumerate}
        \item $u$ lies in the relative interior of a full-dimensional cone of $\mbox{Tree}^{\text{tr}}_n$,
        \item $T$ has exactly $n-1$ internal vertices (including the root), and
        \item all internal vertices of $T$ have degree 3, except the root, which has degree 2.
    \end{enumerate}
\end{defn}

Note that if $u$ is generic, then so is $\lambda \odot u$ for any $\lambda \in \RR$, since tropical multiplication affects only pendant edge lengths.

A tree that is not generic is unresolved.
\begin{defn}\label{defn:unresolved}
    A tree $T$ is unresolved (at internal vertex $v$) if an internal vertex, $v$, has more than two children.
\end{defn}

It is also possible to encounter highly unresolved trees along the tropical line segment between $T_1$ and $T_2$ if certain edge lengths in the trees coincide. Therefore, we also require that $T_1$ and $T_2$ are generic as a pair. We call such a pair a \newword{generic pair}, and in \Cref{defn:generic-pair} we formally define such a pair.

\begin{defn}\label{defn:generic-pair}
    A pair of equidistant trees $T_1, T_2$ on $n$ leaves, with distance vectors $u$ and $v$, respectively, is a \newword{generic pair} if:
    \begin{enumerate}
        \item $u$ and $v$ are each generic, and
        \item for all $x_i, y_i$ distinct internal vertices of $T_i$, $h_{T_1}(x_1) - h_{T_2}(x_2) \neq h_{T_1}(y_1) - h_{T_2}(y_2)$.
    \end{enumerate}
\end{defn}

For a generic pair of trees, the number of turning points is independent of the specific choice of ultrametric.

\begin{lemma}\label{lemma:redo-turn}
    Let $T_1, T_2$ be phylogenetic trees with leaf set $[n]$. Then for any sufficiently general choice of ultrametrics $u^k$ on $T_k$ (i.e. $u^1, u^2$ is a generic pair),
    $$\Lambda(u_1, u_2) = \{ 2( h_{u_1}(x_1) - h_{u_2}(x_2) ) \mid (x_1, x_2) \in \Pi(T_1, T_2) \}$$
\end{lemma}

\begin{proof}
    According to \Cref{defn:line-algo}, the turning points of the tropical line segment between $T_1$ and $T_2$ occur at $u^1 \oplus \lambda_{ij} \odot u^2$, where $\lambda_{ij} = u^1_{ij} - u^2_{ij}$.
    
    We can rewrite the expression above in terms of heights on internal vertices.
    $$\lambda_{ij} = 2(h_{T_1}(x_1) - h_{T_2}(x_2)), \text{ where } x_k = \lca_{T_k}(i,j)$$

    Conversely, given internal vertices $x_1 \in T_1$ and $x_2 \in T_2$, $\mu_{x_1 x_2} = 2(h_{T_1}(x_1) - h_{T_2}(x_2))$ is the scalar for a turning point if and only if $\mu_{x_1 x_2} = \lambda_{ij}$ for some $i, j \in [n]$, which is if and only if there are $i,j \in [n]$ with $x_k = \lca_{T_k}(i,j)$ for $k = 1, 2$.
\end{proof}

\begin{cor}\label{cor:L-is-Pi}
    Let $T_1, T_2$ be phylogenetic trees with leaf set $[n]$. Then for any sufficiently general choice of ultrametrics $u_k$ on $T_k$, $\# \Lambda(u_1, u_2) = \# \Pi(T_1, T_2)$.
\end{cor}

Therefore, the number of turning points between two combinatorial (non-metric) tree structures can be defined as:
\begin{defn}
    The \newword{tropical interchange number} for trees $T_1$, $T_2$, denoted $\mbox{tI}(T_1, T_2)$, is the number of turning points between $T_1$ and $T_2$ for any choice of generic weights $u_1, u_2$ on $T_1, T_2$.
\end{defn}

\section{The Tropical Line and NNI}\label{section:tnni}

Now we come to the first main question of the paper: what kinds of tree topology changes can occur along the tropical line segment between two trees $T_1$ and $T_2$?

It turns out that even for a generic pair of trees, there may be a turning point which is not a single NNI. For example, the endpoint trees in \Cref{fig:fourClade} are a generic pair, but the tropical line segment passes through a tree with an internal vertex with four descendants (and this is independent of the specific choice of metrics on the trees). In this section, we prove that when the endpoint trees are chosen to be sufficiently general, the only moves that can occur are: a single NNI, or a single four clade rearrangement.

\subsection{Possible Turning Points}

The main theorem for this section is the following:

\begin{theorem}\label{thm:turning-pts}
    At each turning point of the tropical line segment between a generic pair of trees, the intermediate tree takes one of the following three forms:
    \begin{enumerate}
        \item a tree that is trivalent except for one internal vertex with 3 children (the middle of an NNI, passing through a co-dimension 1 cell), see \Cref{fig:singleNNI} for an example;
        \item a tree that is trivalent except for one internal vertex with 4 children (the middle of a four clade rearrangement, passing through a co-dimension 2 cell), see \Cref{fig:fourClade} for an example;
        \item a generic tree, (remaining in a top-dimensional cell), see \Cref{fig:noTop} for an example.
    \end{enumerate}
\end{theorem}

\begin{cor}\label{onebad-prop}
    On the tropical line between two generic trees $u$, $v$, there are no trees with an internal vertex with 5 or more children.
\end{cor}

The following definition will be useful in proving the theorem:

\begin{defn}\label{defn:G(uv)}
    For $u,v$ a pair of trees on $n$ leaves, let $G(u \geq v)$ be the graph with vertices $[n]$ and edges $(i,j)$ such that $u_{ij} \geq v_{ij}$.
\end{defn}

Note that $G(u \geq v) \cup G(u \leq v) = K_n$.

\begin{lemma}\label{odd_cycles}
    If $G$ and $H$ are subgraphs of $K_5$, with $G \cup H = K_5$, then at least one of $G$, $H$ contains an odd cycle.
\end{lemma}

\begin{proof}
Assume $G$ and $H$ each contain no odd cycles. Then $G$, $H$ are bipartite graphs on the five vertices. Let $A,B,C$ be distinct vertices of $G$ that share the same color. Then at least one of the pairs $\{ A, B \}$, $\{ A, C \}$, and $\{ B, C \}$ has constant color in $H$. So neither $G$ nor $H$ contains an edge between that pair. This contradicts the assumption that $G \cup H = K_5$.
\end{proof}

\begin{proof}[Proof of \Cref{thm:turning-pts}]
    Suppose for a contradiction that there is some $\lambda \in \Lambda(u, v)$ so that the tree $w = u \oplus \lambda \odot v$ has an internal vertex with 5 or more children. Denote that internal vertex $w_0$. Pick one leaf descending from each child of $w_0$, and call them $1, 2, 3, 4, 5, \ldots$. Then $10 = \binom{5}{2}$ distances coincide:
    $$\max(u_{12}, \lambda + v_{12}) = \max(u_{13}, \lambda + v_{13}) = \cdots = \max(u_{45}, \lambda + v_{45}).$$
    
    \Cref{odd_cycles} shows that one of $G(u \geq \lambda + v)$, $G(u \leq \lambda + v)$ must contain an odd cycle. This means that for $k = 1$ or $k = 2$:
    $$u_{12} = u_{23} = \cdots = u_{2k+1,1}$$
    (or symmetrically for $\lambda \odot v$). Without loss of generality, suppose that the equations involving $u$ hold. Then all but the last distance tells us that $u_0 = \lca_u(i,j)$ for all $i \in \{ 1,3,\ldots,2k+1 \}$ and $j \in \{ 2,4,\ldots,2k \}$, meaning that the odd numbered and even numbered leaves descend from different children of $u_0$. But then the last distance, $u_{2k+1,1}$ says that $1$, $2k+1$ also have least common ancestor $u_0$, so the internal vertex must have at least 3 children, which contradicts the assumption that $u$ is generic.
    
    Therefore, one of $u$, $\lambda \odot v$ is not generic. Since $\lambda \odot v$ is generic if and only if $v$ is generic, this contradicts our assumption that $u$, $v$ are generic (i.e. trivalent) trees. Thus, no intermediate tree on the tropical line between a generic pair $u$, $v$ has an internal vertex with 5 or more children.
\end{proof}

\begin{prop}\label{twobad-prop}
    On the tropical line segment between a generic pair of trees, there are no intermediate trees with two internal vertices that have three or more children each.
\end{prop}

The following lemma will be useful in proving the proposition.

\begin{lemma}\label{lca-lemma}
    Let $i,j,k,l \in [n]$ be leaves, and let $u,v$ be generic equidistant trees on $[n]$ leaves. If $\mbox{LCA}_u(i,j) = \mbox{LCA}_u(k,l)$ and $\mbox{LCA}_v(i,j) = \mbox{LCA}_v(k,l)$, then $\mbox{LCA}_{u \oplus v}(i,j) = \mbox{LCA}_{u \oplus v}(k,l)$.
\end{lemma}

\begin{proof}[Proof of \Cref{lca-lemma}]
    The equality $\mbox{LCA}_u(i,j) = \mbox{LCA}_u(k,l)$ is equivalent to:
    $$d_u(i,j) = d_u(k,l) \geq d_u(i,k), d_u(j,l), d_u(i,l), d_u(j,k).$$
    In $u \oplus v$,
    $$d_{u \oplus v}(i,j) = \max\{ d_u(i,j), d_v(i,j) \} = \max\{ d_u(k,l), d_v(k,l) \} = d_{u \oplus v}(k,l)$$
    and
    $$d_{u \oplus v}(i,j) = \max\{ d_u(i,j), d_v(i,j) \} \geq \max\{ d_u(x,y), d_v(x,y) \} = d_{u \oplus v}(x,y)$$
    for all $xy \in \{ jk, jl, il, jk \}$. Therefore, $\mbox{LCA}_{u \oplus v}(i,j) = \mbox{LCA}_{u \oplus v}(k,l)$.
\end{proof}

\begin{proof}[Proof of \Cref{twobad-prop}]
    Suppose $w = u \oplus (\lambda \odot v)$ has two internal vertices with at least three children each. Name three of the descendants of the first high-degree internal vertex 1, 2, and 3, and name three descendants of the second high-degree internal vertex 4, 5, and 6. Then
    $$w_{12} = w_{13} = w_{23}$$
    and
    $$w_{45} = w_{46} = w_{56}.$$
    If we expand $w_{12} = w_{13} = w_{23}$ in terms of $u$ and $\lambda \odot v$, we see
    \begin{equation}\label{eqn:w-ties-twobad}
        \max(u_{12}, \lambda + v_{12}) = \max(u_{13}, \lambda + v_{13}) = \max(u_{23}, \lambda + v_{23}).
    \end{equation}
    We will show that these equations imply $u_{ij} = \lambda + v_{ij}$ for some $ij \subset \{1,2,3\}$ and $u_{k \ell} = \lambda + v_{k \ell}$ for some $k \ell \subset \{4,5,6\}$, and this implies that $u_{ij} - v_{ij} = \lambda = u_{k\ell} - v_{k\ell}$, which contradicts the assumption that $u,v$ is a generic pair. 
    
    The definition of ultrametric tells us that two of $u_{12}$, $u_{13}, u_{23}$ are equal and greater than the third; without loss of generality, assume $u_{12} = u_{13} \geq u_{23}$. The fact that $u$ is generic implies the inequality is strict: $u_{12} = u_{13} > u_{23}$. Using a similar argument for $\lambda \odot v$, there are three possibilities:
    \begin{gather}
        \lambda + v_{12} = \lambda + v_{13} > \lambda + v_{23} \\
        \lambda + v_{12} = \lambda + v_{23} > \lambda + v_{13} \\
        \lambda + v_{13} = \lambda + v_{23} > \lambda + v_{12}
    \end{gather}
    However, if (2) holds, then $w_{13} = \max\{ u_{13}, \lambda + v_{13} \} > \max\{u_{23}, \lambda + v_{23}\} = w_{23}$, and this contradicts $w_{13} = w_{23}$. Therefore, one of (3) or (4) must hold  (and they are equivalent up to permuting the indices of $u$ and $v$ simultaneously), so without loss of generality assume $\lambda + v_{12} = \lambda + v_{23} > \lambda + v_{13}$.

    Now we must have $w_{13} = u_{13}$. If not, then $w_{13} = \lambda + v_{13} < \lambda + v_{23} \leq w_{23}$, and this contradicts $w_{13} = w_{23}$. Similarly, if $w_{23} \neq \lambda + v_{23}$, then $w_{23} = u_{23} < u_{13} \leq w_{13}$, which again contradicts $w_{23} = w_{13}$. It follows that:
    \begin{gather}\label{eqn:max-line}
        w_{12} = w_{23} = w_{13} = u_{13} = u_{12} > u_{23}, \\
        w_{12} = w_{13} = w_{23} = \lambda + v_{23} = \lambda + v_{12} > \lambda + v_{13} \\
        \implies \lambda + v_{12} = u_{12} = u_{13} = \lambda + v_{23} > u_{23}, \lambda + v_{13}.
    \end{gather}

    In particular, $\lambda + v_{12} = u_{12}$. Following the same argument after replacing $123$ with $456$:
    $$\lambda + v_{45} = u_{45} = u_{46} = \lambda + v_{56} > u_{56}, \lambda + v_{46}.$$
    
    This implies $u_{12} - v_{12} = \lambda = u_{45} - v_{45}$, so $u_{12} - v_{12} = u_{45} - v_{45}$, and this contradicts the assumption that $u, v$ is a generic pair.
\end{proof}

\section{A Very Long Tropical Line}\label{section:bad-example}

Although \Cref{thm:bound} limits the average tropical NNI number to $O(n (\log n)^4)$, specific pairs of trees can have a much larger tropical NNI number. \Cref{fig:trees-long-line} depicts a pair of trees on $[n]$ leaves, which differ by $n-2$ NNI moves, but whose tropical line segment has $\binom{n-1}{2} \sim n^2$ single NNI moves. Denote the trees below by $u^n$ and $v^n$ respectively.

\begin{figure}[!h]
    \centering
    \begin{minipage}{.5\textwidth}
        \centering
        \includegraphics{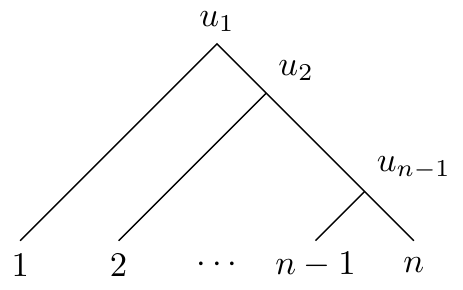}
        \caption*{$u_{ij} = n (n - \min(i,j))$}
    \end{minipage}%
    \begin{minipage}{.5\textwidth}
        \centering
        \includegraphics{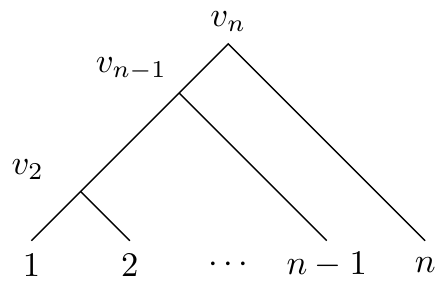}
        \caption*{$v_{ij} = \max(i,j) - 1$}
    \end{minipage}

    \caption{$u^n, v^n$ is a generic pair of trees on $[n]$ leaves with tropical NNI number $\binom{n-1}{2}$. Only $n-2$ NNI moves are needed to transform $u^n$ into $v^n$.}
    \label{fig:trees-long-line}
\end{figure}

\begin{theorem}\label{thm:bad-example}
     Along the tropical line segment from $v^n$ to $u^n$, there are $\binom{n-1}{2}$ single NNIs and no four clade rearrangements.
\end{theorem}

\begin{proof}
    First note that $\left( \lca_u(i,j), \lca_v(i,j) \right) = \left( u_{\min(i,j)}, v_{\max(i,j)} \right)$. Thus, $\# \Pi(u, v) = \binom{n}{2}$. Furthermore, $u,v$ is a generic pair, so $\# \Lambda(u, v) = \# \Pi(u, v)$ and the only possible turning points are those outlined in \Cref{thm:turning-pts}.

    Fix $i < j$. By definition, $u_{ij}$ depends only on $\min(i,j) = i$, and $v_{ij}$ depends only on $\max(i,j) = j$. We will show that at the turning point with scalar $\lambda = u_{ij} - v_{ij}$, there is a NNI move if $|i - j| > 1$ and a generic tree if $|i - j| = 1$. The intermediate tree $w$ is illustrated in \Cref{fig:bad-mids}.

    \begin{figure}
        \centering
        \includegraphics{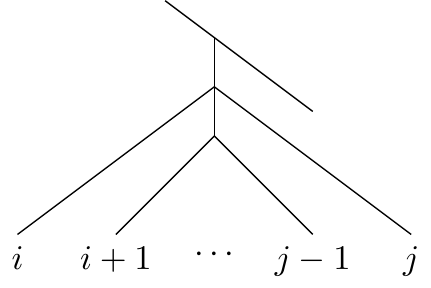}
        \caption{A part of the turning point tree $w = u \oplus (\lambda_{ij} \odot v)$. When $i + 1 = j$, there are no leaves in the subtree between $i$ and $j$, and $w$ is generic.}
        \label{fig:bad-mids}
    \end{figure}
    
    Now let $k \in [i+1, j-1] \cap [n]$. In tree $u$, we have
    $$u_{ij} = u_{ik} > u_{kj},$$
    and in tree $v$, we have
    $$v_{ij} = v_{kj} > v_{ik}.$$
    Recall that we set $\lambda = u_{ij} - v_{ij}$. Let $w = u \oplus (\lambda \odot v)$. Then
    $$u_{ik} = u_{ij} = v_{ij} + \lambda = v_{jk} + \lambda > v_{ik} + \lambda, u_{jk}$$
    which means $w_{ij} = w_{ik} = w_{jk}$, i.e., $\lca_w(i,j) = \lca_w(i,k) = \lca_w(j,k)$. This implies $w$ is non-generic when $|i - j| > 1$.

    Observe that if $k \notin [i, j] \cap [n]$, then $\lca_w(k,i) \neq \lca_w(i,j)$. We can see this by comparing $w_{ki}$ or $w_{jk}$ to $w_{ij}$. If $k < i$, then $u_{ij} < u_{jk}$, and $v_{ij} = v_{jk}$, so 
    $$w_{jk} = \max(u_{jk}, v_{jk} + \lambda) = \max(u_{jk}, v_{ij} + \lambda) = \max(u_{jk}, u_{ij}) = u_{jk} > w_{ij}.$$ 
    An analogous argument shows that $w_{jk} > w_{ij}$ for $k > j$. Therefore, $w$ is non-generic \textit{only when} $|i - j| > 1$. We are just left to rule out a four-clade rearrangement.
    
    If $k_1, k_2 \in [i+1,j-1] \cap [n]$, then
    $$u_{ij} = u_{ik_2} > u_{k_1 k_2}, \text{ and } v_{ij} = v_{k_1 j} > v_{k_1 k_2}$$
    and this implies $w_{k_1 k_2} < w_{ij}$, so $w$ has a NNI and not a four-clade rearrangement. \qedhere
\end{proof}

\textbf{Warning:} The tropical NNI number is not a metric because it does not satisfy the triangle inequality. For example, there is a concatenation of tropical line segments from $v^n$ to $u^n$, over which $n - 2$ single NNI moves occur (the minimum possible number), and this is smaller than the number of NNI moves occurring on the tropical line segment from $v^n$ to $u^n$.

\section{Expected Length of the Tropical Line Segment}\label{section:bound}

Now that we know tree topologies along the tropical line segment change by single NNI or four clade rearrangement moves, we can ask how many NNI moves occur along the tropical line segment between two randomly chosen trees. The goal of this section is to prove \Cref{thm:bound} bounding the number of turning points. A pair of randomly chosen phylogenetic trees on $[n]$ leaves means that the topologies of the trees are uniformly randomly chosen from the (finitely many) binary tree topologies on $[n]$ leaves, and the corresponding ultrametrics form a generic pair.

\begin{theorem}\label{thm:bound}
    The expected number of turning points of the tropical line segment between two randomly chosen phylogenetic trees on $n$ leaves is $O\left(n \left(\log(n)\right)^4 \right)$.
\end{theorem}

\begin{cor}\label{cor:avg-NNI-bound}
    The expected tropical NNI number between two randomly chosen phylogenetic trees on $n$ leaves is $O\left(n \left(\log(n)\right)^4 \right)$.
\end{cor}

\subsection{Sample spaces, Notation, and Equivalence of Probabilities}

We want to bound the number of turning points of the tropical line segment between two phylogenetic trees on $n$ leaves. After a series of translations, we reduce this to a question on the space of unlabelled rooted binary planar trees with one marked internal vertex.

\begin{defn}
    A rooted tree is \textbf{planar} if the children of each vertex are assigned an ordering. For rooted binary planar trees, we will refer to the \newword{left child} and \newword{right child} of an internal vertex, and denote the descendants of the left (respectively right) child by $L(v)$ (resp. $R(v)$). Let $\ell(v) = |L(v)|$ and $r(v) = |R(v)|$.
\end{defn}

We will consider the following sample spaces of trees:
\begin{align*}
    \mathcal{T}_{[n]} &= \{ T : T \text{ is a rooted binary tree with } n \text{ labeled leaves} \}. \\
    \mathcal{T}_{[n]}^{\text{planar}} &= \{ T : T \text{ is a planar rooted binary planar trees with } n \text{ labeled leaves} \}. \\
    \mathcal{T}_{[n], v}^{\text{planar}} &= \{ (T, v) : T \in \tb, v \text{ is an internal vertex of } T \}. \\
    \mathcal{T}_{n}^\text{planar} &= \{ T : T \text{ is a rooted binary planar tree with } n \text{ unlabeled leaves} \}. \\
    \mathcal{T}_{n, v}^\text{planar} &= \{ (T, v) : T \in \td, v \text{ is an internal vertex of } T \}. \\
    \mathcal{T}_{n, v}^\text{planar} (a,b) &= \{ (T, v) \in \te : \ell(v) = a, r(v) = b \}.
\end{align*}

We now re-frame the number of turning points in terms of the last tree sample space above. This proposition makes it clear that the number of turning points is reduced from the naive bound of $\binom{n}{2}$ only because different leaf pairs $(i,j)$ can have the same common ancestor in $T$.

\begin{defn}
    Given $A_i, B_i \subseteq [n]$ with $A_i \cap B_i = \emptyset$ and $|A_i| = a_i$, $|B_i| = b_i$, let $Q(a_1, b_1, a_2, b_2)$ be the probability that the following intersections are non-trivial:
    $$A_1 \cap A_2 \neq \emptyset, B_1 \cap B_2 \neq \emptyset.$$
\end{defn}

\begin{prop}\label{prop:E-bound}
    $$E \left(\# \Pi(T_1, T_2) \mid T_k \in \ta \right) \leq 2(n-1)^2 E \left( Q \left( |L(v_1)|, |R(v_1)|, |L(v_2)|, |R(v_2)| \right) \mid \substack{(T_k, v_k) \in \te \\ k = 1,2} \right).$$
\end{prop}


\begin{proof}
    By definition, $(LHS)$ is the probability that a pair of internal vertices $x_1 \in T_1$, $x_2 \in T_2$ are the least common ancestor of some leaves $i, j$ in their respective trees, times the number of pairs $(x_1,x_2)$. There are $n-1$ internal vertices of each tree $T_k$, so there are $(n-1)^2$ possible pairs. So we find that (LHS) is equal to:
    \begin{equation}\label{eqn:E-2}
        (n-1)^2 P( (x_1, x_2) \in \Pi(T_1, T_2) \mid (T_k, x_k) \in \tc, k=1,2).
    \end{equation}

    We can further rephrase $(\ref{eqn:E-2})$ in terms of planar leaf-labelled trees:
    \begin{equation}\label{eqn:E-3}
        (n - 1)^2 P \left( \substack{L(x_1) \cap L(x_2) \neq \emptyset, R(x_1) \cap R(x_2) \neq \emptyset \\ \text{ or } \\ L(x_1) \cap R(x_2) \neq \emptyset, R(x_1) \cap L(x_2) \neq \emptyset} \mid (T_k, x_k) \in \tc \right).
    \end{equation}

    We can simplify (\ref{eqn:E-3}) by using only one of the conditions. Then, (\ref{eqn:E-3}) is bounded by:
    \begin{equation}\label{eqn:E-4}
        2 (n - 1)^2 P \left( L(x_1) \cap L(x_2) \neq \emptyset, R(x_1) \cap R(x_2) \neq \emptyset \mid (T_k, x_k) \in \tc \right).
    \end{equation}

    Finally, we forget about the leaf-labels and introduce the variable $Q$, so $(\ref{eqn:E-4}) = \mbox{RHS}$.
\end{proof}

\subsubsection{Counting Trees}

In this section, we count trees to provide an explicit formula for the expected value bound that we proved in the last subsection.

\begin{lemma}{\cite[Example 5.3.12]{stanley}}\label{lem:Tnp}
    $$|\td| = \frac{1}{n} \binom{2n - 2}{n - 1}.$$
\end{lemma}

\begin{lemma}\label{lem:Tnvp}
    $$|\te| = \frac{n - 1}{n} \binom{2n - 2}{n - 1}.$$
\end{lemma}

\begin{proof}
    A rooted binary tree on $n$ leaves has $n-1$ internal vertices (including the root), so
    $$|\tc| = (n-1) \cdot \frac{1}{n} \binom{2n - 2}{n - 1} = \frac{n-1}{n} \binom{2n - 2}{n - 1}. \qedhere$$
\end{proof}

\begin{lemma}\label{lem:Tab}
    $$|\te(\ell(v) = a, r(v) = b)| = \frac{1}{ab} \binom{2a - 2}{a - 1} \binom{2b - 2}{b - 1} \binom{2(n - a - b)}{n - a - b}.$$
\end{lemma}

\begin{proof}
    \begin{figure}
        \centering
        \includegraphics{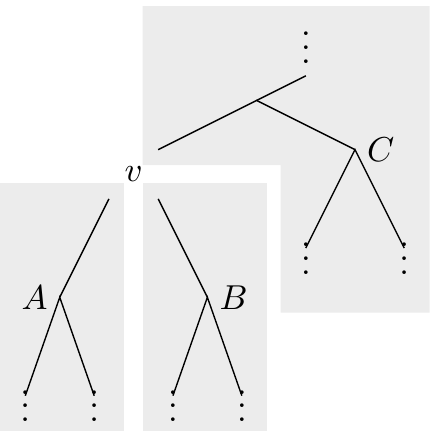}
        \caption{Splitting a binary tree into three binary trees at an internal vertex.}
        \label{fig:tree_split}
    \end{figure}
    
    We want to count the number of binary planar trees on $n$ leaves with one marked internal vertex $v$, which has $a$ left descendants and $b$ right descendants. We can construct such a tree by gluing together three binary planar trees $A$, $B$, and $C$ on $a$, $b$, and $c = (n - a - b) + 1$ leaves respectively, as in \Cref{fig:tree_split}. Specifically, we pick a leaf of tree $C$ to be the marked vertex $v$, then attach tree $A$ (tree $B$) as the left (resp. right) descendants of $v$.

    It follows that:
    \begin{align*}
        |\te(a, b)| &= |\mathcal{T}^\text{planar}_a| |\mathcal{T}^\text{planar}_b| (n - a - b + 1) |\mathcal{T}^\text{planar}_{n - a - b + 1}| && \text{by \Cref{lem:Tnp}} \\
        &= \frac{1}{a} \binom{2a - 2}{a - 1} \frac{1}{b} \binom{2b - 2}{b - 1} \frac{(n - a - b + 1)}{n - a - b + 1} \binom{2(n - a - b)}{n - a - b} \\
        &= \frac{1}{ab} \binom{2a - 2}{a - 1} \binom{2b - 2}{b - 1} \binom{2(n - a - b)}{n - a - b}. && \qedhere
    \end{align*}
\end{proof}

\subsection{Bounding the Sum}

Let $P(a, b; n)$ be the probability of picking a tree-vertex pair $(T, v)$ from $\te$ with $L(v) = a$, $R(v) = b$. In \Cref{prop:E-bound}, we bound the number we want to compute (average tropical NNI number over pairs of trees on $[n]$ leaves) by the following expectation, which we have written in terms of $P(a_i, b_i; n)$ and $Q(a_1, b_1, a_2, b_2; n)$.
\begin{equation}\label{eq:the-sum}
     S_n := \sum_{\substack{a_i + b_i + c_i = n \\ a_i, b_i \geq 1, c_i \geq 0 }} Q(a_1, b_1, a_2, b_2; n) P(a_1, b_1; n) P(a_2, b_2; n)
\end{equation}

We will work to bound (\ref{eq:the-sum}) by bounding $P(a, b; n)$ and $Q(a_1, b_1, a_2, b_2; n)$. It will be useful to have some bounds on the central binomial coefficients.

\begin{prop}[\cite{robjohn}]\label{prop:robjohn-bound}
    \begin{equation*}
        \displaystyle\frac{4^n}{\sqrt{\pi\!\left(n+\frac13\right)}}\le\binom{2n}{n}\le\frac{4^n}{\sqrt{\pi\!\left(n+\frac14\right)}}
    \end{equation*}
\end{prop}

\begin{lemma}\label{lem:Tn-bound}
    For $n \geq 1$, we have the bounds
    $$\frac{4^{n-1}}{n \sqrt{\pi \left( n - \frac{2}{3} \right)}} \leq \left| \td \right| \leq \frac{4^{n-1}}{n \sqrt{\pi\left( n - \frac{3}{4} \right)}}.$$
    $$\frac{4^{n-1} (n-1)}{n \sqrt{\pi \left( n - \frac{2}{3} \right)}} \leq \left| \te \right| \leq \frac{4^{n-1} (n-1)}{n \sqrt{\pi\left( n - \frac{3}{4} \right)}}.$$
\end{lemma}

\begin{proof}
    The bounds are obtained by plugging the bounds from \Cref{prop:robjohn-bound} into the formula we found for $\left| \td \right|$ in \Cref{lem:Tnp}.
    $$\left| \td \right|  = \frac{1}{n} \binom{2n - 2}{n - 1} \leq \frac{4^{n-1}}{n \sqrt{\pi\left( n - 1 + \frac{1}{4} \right)}} = \frac{4^{n-1}}{n \sqrt{\pi\left( n - \frac{3}{4} \right)}}$$
    $$\left| \td \right| = \frac{1}{n} \binom{2n - 2}{n - 1} \geq \frac{4^{n-1}}{n \sqrt{\pi \left( n-1+\frac13 \right)}} = \frac{4^{n-1}}{n \sqrt{\pi \left( n - \frac{2}{3} \right)}}$$

    The bounds for $\left| \te \right|$ are obtained through an analogous argument, observing that as we saw in \Cref{lem:Tnvp} $\left| \te \right| = (n-1) \left| \td \right|$.
\end{proof}

\begin{cor}\label{cor:Pab-bound}
    $$P(a, b; n) \leq \frac{1}{2\pi} \frac{\sqrt{n}}{a \left( a - \frac{3}{4} \right)^{\frac{1}{2}} b \left( b - \frac{3}{4} \right)^{\frac{1}{2}} \left( c + \frac{1}{4} \right)^{\frac{1}{2}}}$$
\end{cor}

\begin{proof}
    We can express $P(a, b; n)$ as a fraction of the number of tree-vertex pairs $(T, v) \in \te$ with $\ell(v) = a$, $r(v) = b$ over the number of all trees in $\te$.
    $$P(a, b; n) = \frac{\left| \te(a,b) \right|}{\left| \te \right|} = \frac{\left|\mathcal{T}_a^{\text{planar}}\right| \left| \mathcal{T}_b^{\text{planar}} \right| (c+1) \left| \mathcal{T}_{c+1}^{\text{planar}} \right|}{\left| \te \right|}$$
    Applying the bounds from \Cref{lem:Tn-bound} to the expression above yields the claimed inequality.
    $$P(a, b; n) \leq \frac{1}{4\pi} \frac{(n - 1)\sqrt{n}}{a \left( a - \frac{3}{4} \right)^{\frac{1}{2}} b \left( b - \frac{3}{4} \right)^{\frac{1}{2}} \left( c + \frac{1}{4} \right)^{\frac{1}{2}} n} \leq \frac{1}{2\pi} \frac{\sqrt{n}}{a \left( a - \frac{3}{4} \right)^{\frac{1}{2}} b \left( b - \frac{3}{4} \right)^{\frac{1}{2}} \left( c + \frac{1}{4} \right)^{\frac{1}{2}}}$$
\end{proof}

\begin{lemma}
    Let $A_1$ and $A_2$ be subsets of $[n]$, chosen uniformly at random from all subsets of size $a_1$ and $a_2$ respectively. Then,
    $$P(A_1 \cap A_2 \neq \emptyset) \leq \frac{a_1 a_2}{n}$$
\end{lemma}

\begin{proof}
    The probability that $1 \in A_1 \cap A_2$ is $(a_1/n) \cdot (a_2/n)$. Thus, the probability that at least one element of $[n]$ lies in the intersection of $A_1, A_2$ is at most $n (a_1/n) (a_2/n) = a_1 a_2/n$.
\end{proof}

\begin{cor}\label{cor:Q-bound}
    $$Q(a_1, b_1, a_2, b_2) \leq \min\left( \frac{a_1 a_2}{n}, \frac{b_1 b_2}{n} \right) \leq \frac{(a_1 a_2 b_1 b_2)^{1/2}}{n}$$
\end{cor}



\begin{defn}
    $$\tilde{Q}(a_1, b_1, a_2, b_2; n) := \frac{(a_1 a_2 b_1 b_2)^{1/2}}{n}.$$
    $$\tilde{Q}_0(a, b; n) := \left( \frac{ab}{n} \right)^{\frac{1}{2}}.$$
\end{defn}

\Cref{cor:Q-bound} says $Q(a_1, b_1, a_2, b_2; n) \leq \tilde{Q}(a_1, b_1, a_2, b_2; n) = \tilde{Q}_0(a_1, b_1; n) \tilde{Q}_0(a_2, b_2; n)$. We are not yet ready to compute the whole sum in (\ref{eq:the-sum}), since we do not have bounds for $P(a, b; n)$ in some edge cases, but we can bound a significant subsum.

\begin{prop}\label{prop:new-sum}
    \begin{equation}
        S_n \leq \left( \phantom. \sum_{\substack{a + b \leq n \\ a, b \geq 1}} \frac{1}{\left( a - \frac{1}{2} \right) \left( b - \frac{1}{2} \right) \left( c + \frac{1}{4} \right)^{\frac{1}{2}}} \right)^2.
    \end{equation}
\end{prop}

\begin{proof}
    We plug in the bounds we computed in \Cref{cor:Pab-bound} and \Cref{cor:Q-bound} and simplify. The key idea here is that the symmetry of the upper bound for $Q$ allows us to express the sum in six variables as the square of a sum in three variables.

    \begin{align*}
        S_n &= \sum_{\substack{a_i + b_i + c_i = n \\ a_i, b_i \geq 1, c_i \geq 0}} Q(a_1, b_1, a_2, b_2) P(a_1, b_1; n) P(a_2, b_2; n) \\
        &\leq \sum_{\substack{a_i + b_i + c_i = n \\ a_i, b_i \geq 1, c_i \geq 0}} \tilde{Q}(a_1, b_1, a_2, b_2; n) P(a_1, b_1; n) P(a_2, b_2; n) \\
        &= \left( \phantom. \sum_{\substack{a + b \leq n \\ a, b \geq 1}} \tilde{Q}_0 (a, b; n) P(a, b; n) \right)^2
    \end{align*}

    Now we expand and bound the terms of the sum in the square above using the bound on $P(a,b;n)$ from \Cref{cor:Pab-bound}.
    \begin{align*}
        \tilde{Q}_0(a, b; n) P(a, b; n) &\leq \frac{a^{\frac12} b^{\frac12}}{\sqrt{n}} \frac{\sqrt{n}}{a \left( a - \frac{3}{4} \right)^{\frac{1}{2}} b \left( a - \frac{3}{4} \right)^{\frac{1}{2}} \left( c + \frac{1}{4} \right)^{\frac{1}{2}}} \\ &= \frac{1}{a^{\frac12} \left( a - \frac{3}{4} \right)^{\frac{1}{2}} b^{\frac12} \left( b - \frac{3}{4} \right)^{\frac{1}{2}} \left( c + \frac{1}{4} \right)^{\frac{1}{2}}}
    \end{align*}

    We can simplify further by applying the AM-GM inequality:
    $$\frac{1}{a^{\frac12} \left( a - \frac{3}{4} \right)^{\frac{1}{2}}} \leq \frac{1}{a - \frac12}$$

    Thus we derive the following bound for $S_n$.
    $$S_n \leq \left( \phantom. \sum_{\substack{a + b \leq n \\ a, b \geq 1}} \frac{1}{\left( a - \frac{1}{2} \right) \left( b - \frac{1}{2} \right) \left( c + \frac{1}{4} \right)^{\frac{1}{2}}} \right)^2.$$
\end{proof}

The following lemma will be useful in completing the computation of the sum on the right hand side of \Cref{prop:new-sum}.
\begin{lemma}\label{lem:prod-to-plus}
    $$\sum_{\substack{a + b = d \\ a, b \geq 1}} \frac{1}{\left( a - \frac{1}{2} \right) \left( b - \frac{1}{2} \right)} \leq \frac{2 (2 + \log \left( d - \frac32 \right))}{d - 1}.$$
\end{lemma}

\begin{proof}
    First, when $a + b = d$ we have:
    $$\frac{1}{\left( a - \frac{1}{2} \right) \left( b - \frac{1}{2} \right)} = \frac{1}{d - 1} \left( \frac{1}{ a - \frac{1}{2} } + \frac{1}{ b - \frac{1}{2} } \right).$$
    Also,
    $$\sum_{a = 1}^{d-1} \frac{1}{a - \frac12} = 2 + \frac{1}{2 - \frac12} + \cdots + \frac{1}{d - \frac32} \leq 2 + \log\left( d - \frac32 \right).$$
    It follows that
    $$\sum_{\substack{a + b = d \\ a, b \geq 1}} \frac{1}{\left( a - \frac{1}{2} \right) \left( b - \frac{1}{2} \right)} = \frac{1}{d - 1} \sum_{\substack{a + b = d \\ a, b \geq 1}} \frac{1}{a - \frac12} + \frac{1}{b - \frac12} = \frac{2}{d - 1} \sum_{a = 1}^{d-1} \frac{1}{a - \frac12} \leq \frac{2 (2 + \log \left( d - \frac32 \right))}{d - 1}.$$
\end{proof}

Recall the following inequality.
\begin{lemma}[Chebyshev's Sum Inequality, \cite{inequalities}]
    If $a_n$ is a decreasing sequence, and $b_n$ is an increasing sequence, then
    $$\frac{1}{n} \sum_{k=1}^n a_k b_k \leq \left( \frac{1}{n} \sum_{k=1}^n a_k \right) \left( \frac{1}{n} \sum_{k=1}^n b_k \right)$$
\end{lemma}

\begin{prop}\label{prop:bound-Sn}
    $\sqrt{S_n}$ is $O\left( (\log(n))^2/n^{\frac12} \right)$.
\end{prop}

\begin{proof}
    We first bound the sum inside the square by applying \Cref{lem:prod-to-plus}.
    \begin{align*}
        \sum_{\substack{a + b \leq n \\ a, b \geq 1}} \frac{1}{\left( a - \frac{1}{2} \right) \left( b - \frac{1}{2} \right) \left( c + \frac{1}{4} \right)^{\frac{1}{2}}} &= \sum_{c = 0}^{n-2} \left( c + \frac{1}{4} \right)^{-\frac{1}{2}} \sum_{\substack{a + b = n - c \\ a, b \geq 1}} \frac{1}{\left( a - \frac{1}{2} \right) \left( b - \frac{1}{2} \right)} \\
        &\leq \sum_{c=0}^{n-2} \frac{2 \left( 2 + \log\left( n - c - \frac32 \right) \right)}{\left( c + \frac{1}{4} \right)^{\frac{1}{2}} (n - c - 1)} \\
        &\leq \sum_{c=0}^{n-2} \frac{2 \left( 2 + \log\left( n - c - 1 \right) \right)}{\left( c + \frac{1}{4} \right)^{\frac{1}{2}} (n - c - 1)} && (\ast)
    \end{align*}

    The sequence $\left( c + \frac{1}{4} \right)^{-\frac{1}{2}}$ is decreasing in $c$, and the sequence $(2 + \log(n - c - 1))/(n - c - 1)$ is increasing in $c$ on the interval $\left[0, n - \frac1e - 1 \right]$. Therefore, we can apply Chebyshev's sum inequality \cite{inequalities} to bound the sum in $(\ast)$.
    \begin{align*}
        \sum_{c=0}^{n-2} \frac{1}{\left( c + \frac{1}{4} \right)^{\frac12}} \frac{2 \left( 2 + \log\left( n - c - 1 \right) \right)}{(n - c - 1)} &\leq \frac{1}{n-1} \left( \sum_{c = 0}^{n-2} \left( c + \frac{1}{4} \right)^{-\frac{1}{2}} \right) \left( \sum_{c = 0}^{n-2} \frac{2(2 + \log(n - c - 1))}{(n - c - 1)} \right) \\
        &\sim \frac{1}{n-1} \left( \sum_{c = 0}^{n-2} \left( c + \frac{1}{4} \right)^{-\frac{1}{2}} \right) \left( \sum_{c = 0}^{n-2} \frac{\log(n - c - 1)}{(n - c - 1)} \right) \\
        &\sim \frac{n^{\frac12} (\log(n))^2}{n-1} \sim \frac{(\log(n))^2}{n^{\frac12}}
    \end{align*}

    In the third line, we use asymptotics derived by thinking of the sums are left-hand or right-hand Riemann sums of monotone functions.
    \begin{equation}\label{eqn:c14-riemann-bounds}
        \int_0^{n-1} \left( x + \frac14 \right)^{-\frac12} dx \leq \sum_{c = 0}^{n-2} \left( c + \frac{1}{4} \right)^{-\frac{1}{2}} \leq 2 + \int_0^{n-2} \left( x + \frac14 \right)^{-\frac12} dx \sim n^{\frac12}
    \end{equation}
    \begin{equation}\label{eqn:lognc1-riemann-bounds}
        \small \int_{-1}^{n-2} \frac{\log(n - x - 1)}{n - x - 1} dx \leq \sum_{c = 0}^{n-2} \frac{\log(n - c - 1)}{(n - c - 1)} \leq \int_0^{n-2} \frac{\log(n - x - 1)}{n - x - 1} dx \sim (\log(n))^2 \qedhere
    \end{equation}
\end{proof}

\begin{proof}[Proof of \Cref{thm:bound}]
    The results in this section prove the following inequality:
    \begin{equation*}\label{eqn:avg-NNI-S-ineq}
        E\left(\# \Pi(T_1, T_2) \mid T_1, T_2 \in \ta \right) \leq 2 (n-1)^2 S_n.
    \end{equation*}
    Furthermore, $S_n$ is $O\left(\left(\log(n)\right)^4/n \right)$ by \Cref{prop:bound-Sn}. It follows that $(n-1)^2 S_n$ is $O\left(n \left(\log(n)\right)^4 \right)$, as claimed.
\end{proof}

\section{Conclusion}

This paper classifies the tree topology changes that can occur along the tropical line segment between two trees, which is an important step towards interpreting statistical learning models in the space of ultrametrics. We also provide a bound for the average number of topology changes on the tropical line segment between two trees, which allows us to roughly compare distance in the NNI graph to the number of tree topology changes occurring along the tropical line segment. However, this bound is likely not tight. In particular, we suspect that the bound we employ for $Q(a_1, b_1, a_2, b_2; n)$ in \Cref{cor:Q-bound} can be improved. Nonetheless, this is a step towards comparing tree topology changes over geodesics in different realizations of $\mbox{Tree}_n$.

\section*{Acknowledgements}

I am grateful to my advisor, David Speyer, and to Ruriko Yoshida for all their guidance and support. I would also like to thank the anonymous reviewers for the detailed comments, and Anna Brosowsky and Sameer Kailasa for the helpful conversations. This material is based on work supported by the National Science Foundation Graduate Research Fellowship under Grant No. DGE-1841052, and by the National Science Foundation under Grant No. 1855135.

\bibliography{ref}

\begin{thebibliography}{10}

\bibitem{bhv}
Louis~J. Billera, Susan~P. Holmes, and Karen Vogtmann.
\newblock Geometry of the space of phylogenetic trees.
\newblock {\em Adv. in Appl. Math.}, 27(4):733--767, 2001.

\bibitem{speyer2004tropical}
David Speyer and Bernd Sturmfels.
\newblock The tropical {G}rassmannian.
\newblock {\em Adv. Geom.}, 4(3):389--411, 2004.

\bibitem{berg_ak}
Federico Ardila and Caroline~J. Klivans.
\newblock The {B}ergman complex of a matroid and phylogenetic trees.
\newblock {\em J. Combin. Theory Ser. B}, 96(1):38--49, 2006.

\bibitem{kim2000slicing}
Junhyong Kim.
\newblock Slicing hyperdimensional oranges: the geometry of phylogenetic
  estimation.
\newblock {\em Molecular phylogenetics and evolution}, 17(1):58--75, 2000.

\bibitem{moulton_steel_oranges}
Vincent Moulton and Mike Steel.
\newblock Peeling phylogenetic ‘oranges’.
\newblock {\em Advances in Applied Mathematics}, 33(4):710--727, 2004.

\bibitem{gill2008regular}
Jonna Gill, Svante Linusson, Vincent Moulton, and Mike Steel.
\newblock A regular decomposition of the edge-product space of phylogenetic
  trees.
\newblock {\em Advances in Applied Mathematics}, 41(2):158--176, 2008.

\bibitem{wald}
M.~K. Garba, T.~M.~W. Nye, J.~Lueg, and S.~F. Huckemann.
\newblock Information geometry for phylogenetic trees.
\newblock {\em J. Math. Biol.}, 82(3):Paper No. 19, 39, 2021.

\bibitem{rudy-fw}
Bo~Lin and Ruriko Yoshida.
\newblock Tropical {F}ermat-{W}eber points.
\newblock {\em SIAM J. Discrete Math.}, 32(2):1229--1245, 2018.

\bibitem{trop-pca}
Ruriko Yoshida, Leon Zhang, and Xu~Zhang.
\newblock Tropical principal component analysis and its application to
  phylogenetics.
\newblock {\em Bull. Math. Biol.}, 81(2):568--597, 2019.

\bibitem{myown}
Ruriko Yoshida and Shelby Cox.
\newblock Tree topologies along a tropical line segment.
\newblock {\em Vietnam Journal of Mathematics}, pages 1--25, 2022.

\bibitem{DasGupta2016}
Bhaskar DasGupta, Xin He, Ming Li, John Tromp, and Louxin Zhang.
\newblock {\em Nearest Neighbor Interchange and Related Distances}, pages
  1402--1405.
\newblock Springer New York, New York, NY, 2016.

\bibitem{Develin2004}
Mike Develin and Bernd Sturmfels.
\newblock Tropical convexity.
\newblock {\em Documenta Mathematica}, 9:1--27, 2004.

\bibitem{textbook}
Diane Maclagan and Bernd Sturmfels.
\newblock {\em Introduction to tropical geometry}, volume 161 of {\em Graduate
  Studies in Mathematics}.
\newblock American Mathematical Society, Providence, RI, 2015.

\bibitem{NNIorigin1}
G~William Moore, M~Goodman, and J~Barnabas.
\newblock An iterative approach from the standpoint of the additive hypothesis
  to the dendrogram problem posed by molecular data sets.
\newblock {\em Journal of theoretical biology}, 38(3):423--457, 1973.

\bibitem{NNIorigin2}
David~F Robinson.
\newblock Comparison of labeled trees with valency three.
\newblock {\em Journal of Combinatorial Theory, Series B}, 11(2):105--119,
  1971.

\bibitem{NNIdistances}
Bhaskar DasGupta, Xin He, Tao Jiang, Ming Li, John Tromp, and Louxin Zhang.
\newblock On distances between phylogenetic trees.
\newblock In {\em SODA}, volume~97, pages 427--436. Citeseer, 1997.

\bibitem{someNNInotes}
Ming Li, John Tromp, and Louxin Zhang.
\newblock Some notes on the nearest neighbour interchange distance.
\newblock In {\em International Computing and Combinatorics Conference}, pages
  343--351. Springer, 1996.

\bibitem{stanley}
Richard~P. Stanley.
\newblock {\em Enumerative combinatorics. {V}ol. 2}, volume~62 of {\em
  Cambridge Studies in Advanced Mathematics}.
\newblock Cambridge University Press, Cambridge, 1999.
\newblock With a foreword by Gian-Carlo Rota and appendix 1 by Sergey Fomin.

\bibitem{robjohn}
robjohn (https://math.stackexchange.com/users/13854/robjohn).
\newblock Elementary central binomial coefficient estimates.
\newblock Mathematics Stack Exchange.
\newblock URL:https://math.stackexchange.com/q/932509 (version: 2019-10-07).

\bibitem{inequalities}
G.~H. Hardy, J.~E. Littlewood, and G.~P\'{o}lya.
\newblock {\em Inequalities}.
\newblock Cambridge Mathematical Library. Cambridge University Press,
  Cambridge, 1988.
\newblock Reprint of the 1952 edition.

\end{thebibliography}
\bibliographystyle{unsrt}

\end{document}